\documentclass[12pt,reqno,twoside]{amsart}

\usepackage{amssymb}
\usepackage{amsmath}

\usepackage{pstricks, pst-3d}
\newcommand{\mm}{\mathfrak m}

\newcommand{\II}{\mathfrak I}
\newcommand{\Z}{\mathbb{Z}}
\newcommand{\R}{\mathbb{R}}
\newcommand{\N}{\mathbb{N}}
\newcommand{\Q}{\mathbb{Q}}

\newcommand{\Fc}{\mathcal{F}}

\newcommand{\Mcc}{\mathcal{M}}
\newcommand{\Ncc}{\mathcal{N}}

\DeclareMathOperator{\pnt}{\raise 0.5mm \hbox{\large\bf.}}

\DeclareMathOperator{\relint}{int}

\DeclareMathOperator{\Tor}{Tor}
\DeclareMathOperator{\Ker}{Ker}

\DeclareMathOperator{\ini}{in}

\def\+#1{\relax\ifmmode\if\noexpand #1\relax \mathop{\kern
    0pt^+{#1}}\nolimits\else \kern 0pt^+\!#1 \fi\else$^*$#1\fi}

\let\phi=\varphi
\let\:=\colon

\newtheorem{thm}{\bf Theorem}[section]
\newtheorem{lem}[thm]{\bf Lemma}
\newtheorem{cor}[thm]{\bf Corollary}
\newtheorem{prop}[thm]{\bf Proposition}
\newtheorem{conj}[thm]{\bf Conjecture}

\theoremstyle{definition}
\newtheorem{defn}[thm]{\bf Definition}

\newtheorem{rem}[thm]{\bf Remark}
\newtheorem{ex}[thm]{\bf Example}

\theoremstyle{plain}
\newtheorem*{thm*}{Theorem}

\title{On the Koszul property of toric face rings}
\author{Dang Hop Nguyen}
\address{Am Herrenberge 11, Appartment 419, 07745 Jena, Germany}
\email{nhop@uos.de}

\thanks{The author has been supported by the graduate school ``Combinatorial Structures in Algebra and Topology'' at the University of Osnabr\"uck, Germany and (partly) by the Vigoni project for a stay in Genoa in March 2011.}
\begin{document}

\begin{abstract}
Toric face rings is a generalization of the concepts of affine monoid rings and Stanley-Reisner rings. We consider several properties which 
imply Koszulness for toric face rings over a field $k$. Generalizing works of Laudal, Sletsj\o{}e and Herzog et al., graded Betti numbers of $k$ over the toric face rings are computed, and a characterization of Koszul toric face rings is provided. We investigate a conjecture suggested by 
R\"{o}mer about the sufficient condition for the Koszul property. The conjecture is inspired by Fr\"{o}berg's theorem on the Koszulness of quadratic 
squarefree monomial ideals. Finally, it is proved that initially Koszul toric face rings are affine monoid rings.
\end{abstract} 

\maketitle

\section{Introduction}
\label{intro} 
Let $k$ be a fixed field. Let $\Sigma$ be a {\em rational pointed fan} in $\R^d$ ($d \geq 1$), i.e.~ $\Sigma$ is a
 collection of rational pointed cones in $\R^d$ satisfying two conditions. Firstly, if $C\in \Sigma$ and $D$ is a face of $C$,
 then $D\in \Sigma$. Secondly, if $C, C' \in \Sigma$ then $C\cap C'$ is either empty or a common face of $C$ and $C'$. 

A {\em monoidal complex} $\mathcal{M}$ supported on $\Sigma$ is a collection of affine monoids
$M_C$ indexed by elements $C$ of $\Sigma$, such that $M_C$ generates $C$ and the following compatibility condition is fulfilled:
if $D \subseteq C \in \Sigma$, then $M_D = M_C \cap D$. 

Given $\Sigma$ and $\Mcc$ as above, define the toric face ring $k[\mathcal{M}]$ of $\mathcal{M}$ over $k$ as follow. The free 
$k$-module $k[\mathcal{M}]$ has a basis $\{t^a:a\in \cup_{C\in \Sigma}M_C\}$. The multiplication is defined for basis elements by the rule
\[
t^a\cdot t^b=t^{a+b}
\]
if $a,b$ are contained in $M_C$ for a cone $C\in \Sigma$. Otherwise, we let $t^a\cdot t^b=0$. Expanding the multiplication rule linearly, we get 
the multiplication of $k[\Mcc]$.

Toric face rings include affine monoid rings and Stanley-Reisner rings as special cases. Indeed, if the fan $\Sigma$ is the face poset of one cone, 
then $k[\Mcc]$ is an affine monoid ring. On the other hand, if all the cones of $\Sigma$ are simplicial, and $M_C = C\cap \Z^d$ is generated by exactly $\dim C$ elements for every 
$C\in \Sigma$, then $k[\Mcc]$ is a Stanley-Reisner ring. Starting with the work of Stanley \cite{Sta1}, several authors have considered toric face 
rings  \cite{BG}, \cite{BKR}, \cite{IR}, \cite{OkaYan}. For an algebraic treatment of affine monoid rings and Stanley-Reisner rings, see Bruns-
Herzog \cite{BH} or Bruns-Gubeladze \cite{BG}; for a more combinatorial treatment of Stanley-Reisner rings see Stanley \cite{Sta2}.

Let $R$ be a homogeneous affine $k$-algebra. We say that $R$ is a {\em Koszul algebra} over $k$ if $k$ has a $R$-linear resolution. This is 
equivalent to the condition that the Betti numbers $\beta ^R_{i,j}(k)=0$ for all $i\neq j$. Assume that $R$ is a Koszul $k$-algebra. Then $R$ is
 defined by quadratic relations over some polynomial ring. Fr\"oberg \cite{Fr1} proves that if $R$ is a Stanley-Reisner ring 
defined by quadratic monomial relations, then $R$ is Koszul. 

Closely related to the Koszul property is the {\em G-quadratic} property. Let $R=S/I$ be a presentation of $R$, where $S$ is a standard graded polynomial ring over $k$, and $I$ is a homogeneous ideal of $S$. Then $R$ is said to be {\em G-quadratic} if its defining ideal $I$ has a quadratic Gr\"obner basis with respect to some term order of $S$. It is well-known that $G$-quadratic algebras over $k$ are Koszul.

Herzog, Hibi and Restuccia \cite{HHR} defined strongly Koszul algebras. A homogeneous $k$-algebra is {\em strongly Koszul} if
its irrelevant ideal admits a system of generators of degree $1$, namely $a_1,\ldots,a_n$, such that
for all increasing sequence $1 \le i_1 < \cdots < i_j\le n$, the ideal $(a_{i_1},\ldots,a_{i_{j-1}}):a_{i_j}$
is generated by a subset of $a_1,\ldots,a_n$. Motivated by the notion of strongly Koszul algebras in \cite{HHR}, Koszul filtrations were introduced
in \cite{CTV} as an useful tool to deduce Koszulness. 
\begin{defn}[Conca, Trung, Valla]
\label{K_filtration}
 A family $\Fc$ of ideals of $R$ is said to be a {\em Koszul filtration} of $R$ if: 
\begin{enumerate}
 \item every ideal of $\Fc$ is generated by linear forms;
 \item the ideal $0$ and the graded maximal ideal belong to $\Fc$;
 \item for every $I \in \Fc$ different from $0$, there exists $J \in \Fc$ strictly contained in $I$ and a
linear form $x \in I$ such that $J+(x) = I$ and $J:I \in \Fc$.
\end{enumerate}
\end{defn}
 Note that a Koszul filtration, if it exists, does not contain the unit ideal. It is known that a ring which has a Koszul filtration must be 
Koszul. This gives another proof to the above result of Fr\"oberg, because in this case, the family of those ideals, each of which is generated by some 
variables, form a Koszul filtration. Moreover, a general set of at most $2n$ points in general linear position has a Koszul filtration \cite[Thm.~2.1]{CTV}.

$R$ is {\em initially Koszul} (abbreviated i-Koszul) with respect to a sequence 
$a_1, \ldots, a_n \in R_1$, if the family of ideals $\mathcal{F} = \{(a_1,\ldots,a_i): i=0,\ldots,n \}$ is a Koszul filtration
 of $R$. Algebras which are i-Koszul  must also be Koszul, and in fact they even have the stronger property of being G-quadratic; see \cite{Bl}, \cite{CRV} for details. See also the survey article \cite{Fr2}.

Our main concern in this paper is to find simple criteria for the various Koszul-like properties of the toric face ring $k[\Mcc]$. To illustrate, let us
look at the special cases of Stanley-Reisner rings and affine monoid rings. In the case of Stanley-Reisner rings, the answers 
are simple. A Stanley-Reisner ring is Koszul if and only if its defining ideal is quadratic (see \cite{Fr1}). Moreover, in this case the ring is
 strongly Koszul \cite[Cor.~2.2]{HHR}. A Stanley-Reisner ring is i-Koszul only if the simplicial complex is a full 
simplex, see \cite[Prop.~2.3]{Bl}. 

In the case of affine monoid rings $R$, we have formulae for the Betti numbers of $k$ over $R$ by the bar resolution \cite{HRW}. For results and problems about the Koszul property of polytopal algebras, see \cite[Chap.~7]{BG} and \cite{BGT}. One can see that the formulae of Betti numbers of $k$ involve infinitely many complicated simplicial homology groups. Hence we find sufficient conditions for the Koszul property of $k[\Mcc]$ 
by looking for situations where the defining ideals of the toric face rings are closed to being quadratic monomial ideals. A reasonable condition is the quadratic condition, which requires that for some system of generators, the monomial part of the defining ideal of $k[\Mcc]$ is quadratic. We will investigate a question of R\"omer, asking if the quadratic condition is satisfied and all the $k[M_C]$ is Koszul for all $C\in \Sigma$ then $k[\Mcc]$ is Koszul, see Conjecture \ref{main_conj}. We are able to prove that the quadratic condition is stable under various natural operations on toric face rings, e.g. ~taking tensor products, fiber products, Segre products, Veronese subrings or multigraded algebra retracts. Furthermore, given the quadratic condition, $k[\Mcc]$ is strongly Koszul or $G$-quadratic if and only if the same thing is true for the subrings $k[M_C]$ for all $C$, see Proposition \ref{quadGr} and Remark \ref{rmstrgKoszul}.
 
The paper is organized as follows. In Section \ref{Toricfacering}, we recall the basic theory of toric face rings and introduce homogeneous toric face rings. In the next section, we prove that the class of homogeneous toric face rings is very natural to consider ring-theoretically, in the sense that it is closed under basic operations of algebras over $k$.  In Section \ref{Bettinumbers}, we compute certain
 graded Betti numbers of $k$ over $k[\Mcc]$. From this computation, we get a characterization of the Koszul property of $k[\Mcc]$.
In Section \ref{Koszul_conj}, we prove that if $k[\Mcc]$ is Koszul, then all $k[M_C]$ are Koszul where $C\in \Sigma$. We introduce the quadratic condition and prove its stability under many different constructions involving homogeneous toric face rings.
It is expected that $k[\Mcc]$ is Koszul if the quadratic condition is satisfied for some system of generators and the monoid ring $k[M_C]$ is Koszul for all facet $C$ of $\Sigma$, see Conjecture \ref{main_conj} of R\"omer. Section \ref{StronglyKoszul} is devoted to the characterization of strong Koszulness of toric face rings. Finally, in Section \ref{InitiallyKoszul}, we prove that homogeneous i-Koszul toric face rings are simply affine monoid rings.

The content of this paper is partially included in the author's dissertation \cite{Nguyen2}.

\section{Notations and background}
\label{Toricfacering}
Let $k$ be a field, $d \geq 1$ a natural number. Denote by $\R_+$ the non-negative real numbers.
\subsection{Embedded toric face rings}
Throughout the paper, we adopt the following notations, which we explain below: $\Sigma$ is a rational pointed fan in $\R^d$, 
$\mathcal{M}$ is a monoidal complex supported on $\Sigma$. The toric face ring of $\Mcc$ over $k$ is $R=k[\Mcc]$. The set $\{a_1,\ldots,a_n\}$ is
 a homogeneous system of generators of $\mathcal{M}$. The ideal $I=I_{\Mcc}$ defines $k[\Mcc]$ as a quotient ring of $k[X_1,\ldots,X_n]$.

A rational pointed fan $\Sigma$ is a collection of rational pointed cones such that:
\begin{enumerate}
\item
if $C\in \Sigma$ and $D$ is a face of $C$ then $D \in \Sigma$;
\item 
for every $C, C' \in \Sigma$, $C\cap C'$ is either a face of both $C$ and $C'$ or empty.
\end{enumerate}

$\Sigma$ is called {\em simplicial} if each of its cones $C$ is generated by linearly independent vectors in $\R^d$.
A maximal element of $\Sigma$ with respect to (w.r.t.) inclusion is called a {\em facet} of $\Sigma$. A one dimensional face of a cone 
of $\Sigma$ is called an {\em extremal ray}.

A monoidal complex $\mathcal{M}$ supported on $\Sigma$ is a collection of affine monoids $M_C$, where $C$ varies in $\Sigma$ such that:
\begin{enumerate}
\item 
$M_C\subseteq C\cap \Z^d$ and $\R_+M_C = C$;
\item
for every $C, D \in \Sigma$ with $D\subseteq C$, $M_{D} = M_C \cap D.$
\end{enumerate}
For instance, after taking $M_C = C\cap \Z^d$ for each $C$ we get a monoidal complex supported on the fan $\Sigma$.

The toric face ring of $\mathcal{M}$ over $k$, denoted by $k[\mathcal{M}]$ is defined as follows. As a $k$-vector space we set
$$
k[\mathcal{M}] = \bigoplus _{a\in \cup _{C\in \Sigma}M_C}kt^a.
$$
The product on basis elements is given by
$$
t^a\cdot t^b = \begin{cases}
                 t^{a+b} & \text {if for some $C\in \Sigma$ both $a$ and $b$ belongs to $M_C$};\\
                       0 & \text{otherwise.}                            
                  \end{cases}
$$
Sometimes we write $a$ instead of the basis element $t^a$ of $k[\mathcal{M}]$. In that case, instead of $t^a\cdot t^b$, we write
 $a\cdot b$, hence in $k[\mathcal{M}]$:
$$
a\cdot b = \begin{cases}
                 a+b & \text {if for some $C\in \Sigma$ both $a$ and $b$ belongs to $M_C$};\\
                       0 & \text{otherwise.}                            
                  \end{cases}
$$
It is known that $k[\mathcal{M}]$ is a reduced $k$-algebra with unit $1=t^0$. Moreover, from the primary decomposition of 
$k[\Mcc]$, we have
$$
\dim k[\mathcal{M}] = \max \{\dim C: C \in \Sigma\},
$$ 
see \cite[Lem.~2.1]{IR}.
An important aspect is that $k[\mathcal{M}]$ inherits the $\Z^d$-grading from the embedding of the monoidal complex. Every $\Z^d$-graded 
component of $k[\mathcal{M}]$ has $k$-dimension less than or equal to $1$.

The two basic examples of toric face rings are Stanley-Reisner rings and affine monoid rings. We recall these examples 
for later usage.
\begin{ex}(Stanley-Reisner rings)
\label{Sta-Rei}
Let $\Delta$ be a simplicial complex on the vertex set $[n] = \{1,\ldots,n\}$. Let $e_1,\ldots,e_n$ be the standard basis vectors
 of $\R^n$. For each face $F$ of $\Delta$, consider the cone $C_F$ generated by the vectors $e_i, i\in F$. It is clear 
that the collection $\Sigma = \{C_F, F\in \Delta \}$ is a rational pointed fan in $\R^n$. For each $F\in \Delta$, choose
$M_{C_F}=C_F \cap \Z^n$. In this manner, we get a monoidal complex $\Mcc$ supported on $\Sigma$. 

The ring $k[\Mcc]$ is the Stanley-Reisner ring $k[\Delta]$. By definition, this is $k[X_1,\ldots,X_n]$ modulo the
 square-free monomial ideal $I_{\Delta}=(\prod_{j\in G}X_j:G\subseteq [n], G \notin \Delta)$.
\end{ex}
\begin{ex}(Affine monoid rings)
\label{asr}
Let $M \subseteq \N^d$ be a finitely generated monoid ($d\ge 1$). Choosing $\Sigma$ to be the face poset of $C= \R_
+M$. For each face $F$ of $C$, let $M_F =M\cap F$. The resulting toric face ring is the affine monoid ring $k[M]$.
\end{ex}
For each cone $C \in \Sigma,$ the affine monoid ring $k[M_C]$ is naturally a subring of $R$. We have natural surjections $R \rightarrow k[M_C]$ defined
 by:
$$
t^a \longmapsto \begin{cases}
                 t^a & \text {if $a$ belongs to $M_C$};\\
                 0 & \text {otherwise.}
                \end{cases}
$$
The composition of the natural inclusion morphism $k[M_C] \rightarrow R$ with the projection $R \rightarrow k[M_C]$ is the identity on $k[M_C]$, 
in other words $k[M_C]$ is an {\em algebra retract} of $R$ for every $C\in \Sigma$.
\begin{defn}
The finite set $\{a_1,\ldots, a_n\}$ is a system of generators of $\mathcal{M}$ if $a_i \in
 \cup _{C\in \Sigma}M_C$ for every $i \in [n]$, and the subset $\{a_1,\ldots,a_n\} \cap M_C$ is a system of generator of $M_C$
 for every $C\in \Sigma$. 
\end{defn}
This system of generators gives a surjection $\phi :S=k[X_1,\ldots,X_n] \rightarrow R$.
 Let $I = \Ker \phi$. For each cone $C$ of $\Sigma$, let $S_C = k[X_i: a_i \in M_C]$, we have 
a map $\phi_C: S_C \rightarrow k[M_C]$ with kernel $I_C=\Ker \phi_C$.

Denote by $\Delta _{\mathcal M}$ the following simplicial complex on the set $[n]$: a subset $F \subseteq [n]$ is a face of $\Delta _
{\mathcal M}$ if and only if there exists some cone $C\in \Sigma$ such that $\{a_j \mid j\in F\} \subseteq M_C$.
\begin{prop}[\cite{BKR}, Prop.~2.3]
\label{def-ideal} Assume that $C_1,\ldots,C_r$ are the facets of $\Sigma$. Then
$$
I = A_\mathcal{M} + \sum \limits _{i=1}^r S\cdot I_{C_i},
$$
where $A_\mathcal{M}$ is generated by square-free monomials $\prod _{j\in G}X_j,$ for which $G \notin \Delta _{\mathcal M}$.
\end{prop}
We call a binomial in $I_{C_i}, i=1,\ldots,r$ a {\em pure binomial}. A binomial $X^a-X^b$ belongs to $I$ if and only if either 
$X^a, X^b\in A_\mathcal{M}$ or $X^a-X^b$ is a pure binomial.

In the following, we call the ideal $A_\mathcal{M}$ the {\em monomial part} of $I$. Denote by $B_{\mathcal M}$ the ideal of $S$ 
generated by the pure binomials in $I$. 

It is the appropriate place to give a simple example of a genuine toric face ring.
\begin{ex}
\label{nontrivialtfr}
Consider the points in $\R^3$ with the following coordinates 
\begin{gather*}
O =(0,0,0),\ A_1 = (2,0,0),\ A_2=(0,2,0),\ A_3 = (0,0,2),\ A_4 = (1,1,0).
\end{gather*}
Consider the fan $\Sigma$ in $\R^3$ with the maximal cones $\normalfont {OA_1A_2}$, $\normalfont {OA_1A_3}$, and
$\normalfont{OA_2A_3}$. Let $\Mcc$ be the monoidal complex supported on $\Sigma$ with the three maximal monoids
 generated by $\{A_1, A_2, A_4\}$, $\{A_1,A_3\}$ and $\{A_2,A_3\}$.

The defining ideal of the toric face ring $k[\Mcc]$ is
$$
I_{\Mcc}= (X_3X_4, X_1X_2X_3)+ (X_1X_2-X_4^2)=(X_1X_2-X_4^2, X_3X_4).
$$
Thus the toric face ring of $\Mcc$ is
$$
k[\Mcc]=k[X_1,X_2,X_3,X_4]/(X_1X_2-X_4^2, X_3X_4).
$$

\bigskip
\bigskip
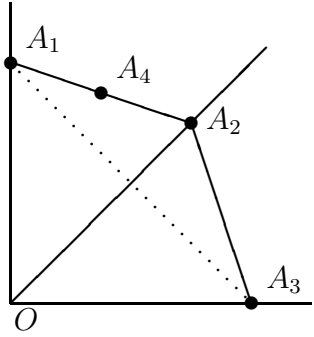
\setlength{\unitlength}{4cm}
\begin{picture}(1,1)
\thicklines
\put(1.5,0){\line(0,1){1}}
\put(1.5,0){\line(1,0){1}}
\put(1.5,0){\line(1,1){0.85}}

\put(1.51,-0.09){$O$}
\put(1.55,0.85){$A_1$}
\put(1.5,0.8){\circle*{0.04}}
\put(2.35,0.05){$A_3$}
\put(2.3,0){\circle*{0.04}}
\put(2.15,0.58){$A_2$}
\put(2.1,0.6){\circle*{0.04}}
\put(1.85,0.75){$A_4$}
\put(1.8,0.7){\circle*{0.04}}

\put(1.5,0.8){\line(3,-1){.6}}
\put(2.3,0){\line(-1,3){.2}}
\multiput(1.5,0.8)(0.03,-0.03){27}{\circle*{0.01}}
\end{picture}
\bigskip
\bigskip

Clearly $k[\Mcc]$ is not a domain, so it is not isomorphic to any monoid ring. 

Now assume that $k[\Mcc]$ is isomorphic to a Stanley-Reisner ring. Observe that $(X_3)$ is a minimal prime ideal of $k[\Mcc]$, see \cite[Lem.~2.1]{IR}. Thus we have $k[\Mcc]/(X_3)$ is isomorphic to a polynomial ring. 

Let $M$ be the monoid generated by $\{A_1, A_2, A_4\}$. We will show that
$$
k[\Mcc]/(X_3) \cong k[M]
$$
is not a polynomial ring. In fact, $k[M]$ is a regular ring only if $M\cong \N^2$ (see, e.g., \cite[Exercise 6.1.11]{BH}). The
 last isomorphism is not possible. Thus $k[\Mcc]$ is not isomorphic to any Stanley-Reisner ring.
\end{ex}
\begin{defn}[Homogeneous system of generators for $\Mcc$]
The set $\{a_1,\ldots,a_n\}$ is said to be a {\em homogeneous system of generators} of $\mathcal{M}$ if for every $C\in \Sigma$, the 
$k$-algebra $k[M_C]$ is {\em minimally} generated in degree 1 by $\{a_1, \ldots, a_n\} \cap M_C$. We call $k[\mathcal{M}]$ a {\em homogeneous toric 
face ring} if $\mathcal{M}$ has a homogeneous system of generators.
\end{defn}
Given a homogeneous system of generators of $\mathcal{M}$, the $\Z$-gradings on the rings $k[M_C]$ where $C\in \Sigma$ induce a 
$\Z$-grading on $k[\mathcal{M}]$. We do not require the $\Z$-grading to be compatible with the existing $\Z^d$-grading, as we
 can use the two gradings separately.
\subsection{Koszul algebras}
\label{pre_Koszul}
A homogeneous $k$-algebra $R$ is a Koszul algebra if and only if $\beta ^R_{i,j}(k)=0$ for all $i\neq j$, where $\beta ^R_{i,j}(k)= \dim _k 
\Tor ^R _i(k,k)_j$ are the Betti numbers of $k$ as a graded $R$-module. In other words, $R$ is Koszul if the module $k=R/R_+$ has a linear resolution 
as an $R$-module.

Let $R=S/I$ be a presentation of $R$, where $S=k[x_1,\ldots,x_n]$ be a standard graded polynomial ring and $I$ a homogeneous ideal of $S$. The following result is well-known; see, e.g., \cite[Prop.~3.13]{BC}.
\begin{thm}
If $I$ has a quadratic Gr\"obner basis with respect to some term order on $S$ then $R=S/I$ is a Koszul algebra.
\end{thm}
When $I$ has a quadratic Gr\"obner basis with respect to some term order of $S$, $R$ is said to be $G$-quadratic (w.r.t.~ that term order), where $G$ stands for Gr\"obner. Hence $G$-quadraticity implies Koszulness. 

Recall that given an inclusion of rings $R\hookrightarrow S$, $R$ is called an algebra retract of $S$ if there is a ring morphism $\phi :S\rightarrow R$ (the retraction map) such that $\phi$ restricts to the identity on $R$. If the rings are $\Z^d$-graded then we also require the ring homomorphisms to preserve the gradings. Koszul algebras are very well-behaved with respect to graded algebra retracts.
\begin{prop}[{\cite[Prop.~1.4]{HHO}}]
\label{Koszulness_under_retracts}
If $R\hookrightarrow S$ is an algebra retract of homogeneous $k$-algebras, then $S$ is Koszul if and only if $R$ is Koszul and considered as an $S$-module via the retraction map, $R$ has a linear free resolution. 
\end{prop}
For further information and results on Koszul algebras, the reader may consult \cite{BFr}, \cite{Fr2}.
\subsection{Multigraded algebra retracts}
\begin{defn}
Let $\Gamma$ be a subfan of $\Sigma$ and $\Mcc_{\Gamma}$ the induced monoidal subcomplex supported on $\Gamma$. We say that $\Mcc_{\Gamma}$ is a 
{\em restricted subcomplex} of $\Mcc$ if for every finite set of elements $z_1,\ldots,z_n$ in $|\Mcc_{\Gamma}|$ such that there is
 a cone $C\in \Sigma$ with the property $z_1,\ldots,z_n \in M_C$, we can also find a cone $D\in \Gamma$ such that 
$z_1,\ldots,z_n \in M_D$.
\end{defn}
For example, given any $C\in \Sigma$, the monoid $M_C$ is a restricted subcomplex of $\Mcc$.

We can classify all $\Z^d$-graded algebra retracts of $k[\Mcc]$ as follow; for completeness, we give a full proof below.
\begin{prop}[{\cite[Prop.~ 4.4]{EH}}]
\label{monoidal_complexes_retracts}
The ring $A$ is a $\Z^d$-graded algebra retract of $k[\Mcc]$ if and only if there is a restricted subcomplex $\Mcc_{\Gamma}$ of $\Mcc$ and a $\Z^d$-graded isomorphism 
$A\cong k[\Mcc_{\Gamma}].$
\end{prop}
\begin{proof}
The ``if" direction: we have to show that $k[\Mcc_{\Gamma}]$ is a $\Z^d$-graded algebra retract of $k[\Mcc]$ for each restricted subcomplex $\Mcc_{\Gamma}$. For this, we observe that there is a natural inclusion  $k[\Mcc_{\Gamma}] \hookrightarrow k[\Mcc]$ mapping $t^a$ to itself for each $a\in |\Mcc_{\Gamma}|$. To check that this is 
indeed an inclusion, we use Proposition \ref{def-ideal}. The binomial relations of $k[\Mcc_{\Gamma}]$ clearly map to zero. The monomial relations also map to zero, since any sequence of elements of $\Mcc_{\Gamma}$ which is not contained in any face of $\Sigma$ also is not contained in any face of $\Gamma$. The composition of 
$k[\Mcc_{\Gamma}] \hookrightarrow k[\Mcc]$ with the natural projection $k[\Mcc] \to k[\Mcc_{\Sigma}]$ is the identity of $k[\Mcc_{\Gamma}]$.

The ``only if" direction: if there is an $\Z^d$-graded algebra retract $A\to k[\Mcc]$ with multigraded inverse $\phi: k[\Mcc] \to A$, then $A\cong k[\Mcc]/L$ where 
$L=\ker \phi$ is a $\Z^d$-graded ideal of $k[\Mcc]$. Now $A$ is reduced since $k[\Mcc]$ is so, therefore $L$ is a radical ideal. From \cite[Lem.~2.1]{IR}, $L$ is generated by elements of $|\Mcc|$ which are not in a subfan $\Gamma$, and this implies $A\cong k[\Mcc_{\Gamma}]$. Similar argument as above shows that $\Mcc_{\Gamma}$ is a restricted 
subcomplex of $\Mcc$. This is our desired claim.
\end{proof}
\section{Basic operations on toric face rings}
\label{operations_on_toric_face_rings}
In this section, we show that the class of homogeneous toric face rings over $k$ is closed under taking Veronese subrings, tensor products,
fiber products and Segre products over $k$. These results show that toric face rings are natural objects to consider, even if we are mainly 
interested in Stanley-Reisner rings. For example, if we take Veronese subrings of a Stanley-Reisner ring, then we do not get 
Stanley-Reisner rings but toric face rings. 

The following result is implicit in \cite{BR}; we give the details for completeness.
\begin{prop}
\label{Veronese_of_toric_face_rings}
For each $m\ge 1$, the Veronese subring $k[\Mcc]^{(m)}$ of a homogeneous toric face ring $k[\Mcc]$ is also a toric face ring.
\end{prop}
\begin{proof}
Consider the Veronese $\Mcc^{(m)}$ of the monoidal complex $\Mcc$ defined as follows. For each $C\in \Sigma$, we let 
$M^{(m)}_C= mM_C:=\{ma:a\in M_C\}$. It is easy to check that the affine monoids $M^{(m)}_C, C\in \Sigma$ form a monoidal complex
$\Mcc^{(m)}$ supported on $\Sigma$. Moreover $k[\Mcc^{(m)}] \cong k[\Mcc]^{(m)}.$ Hence $k[\Mcc]^{(m)}$ is also a toric face ring.
\end{proof}
\begin{prop}
\label{tensor_product_of_toric_face_rings}
For two monoidal complexes $\Mcc, \Ncc$ supported on the fans $\Sigma \subseteq \R^d$, $\Gamma \subseteq \R^{e}$, the tensor product 
$k[\Mcc] \otimes_k k[\Ncc]$ is also a toric face ring.
\end{prop}
\begin{proof}
Assume that the affine monoids of $\Mcc$ are $M_C$ where $C\in \Sigma$, and the monoids of $\Ncc$ are $N_D$ where $D\in \Gamma$.

We will construct a {\em join} monoidal complex $J(\Mcc,\Ncc)$ of $\Mcc$ and $\Ncc$. This construction is the same as building a pyramid after given 
a base polytope and an apex when $\Ncc=\N \subseteq \R$. Embed $\R^d$ and $\R^{e}$ into $\R^{d+e}$ as vector 
subspaces such that they intersect only at the origin.

For each $C\in \Sigma, D \in \Gamma$, let $J(M_C,N_D)$ be the submonoid of $\Z^{d+e}$ generated by $M_C$ and $N_D$, considered as subsets of 
$\Z^{d+e}$. Moreover, let $J(C,D)$ be the convex hull of $C$ and $D$. Then $J(C,D)$ is a cone in $\R^{d+e}$. 

We observe that the cones $J(C,D)$ form a fan $J(\Sigma, \Gamma)$ in $\R^{d+e}$. To see this, choose arbitrary $C,C'\in \Sigma, D,D'\in \Gamma$. 
Since $\R^d \cap \R^e=\{0\}$, we have $\R^d \cap J(C,D)=C$ and $\R^e\cap J(C,D)=D$. So $J(C,D)$ is a face of the cone $J(C',D')$ if and only
 if $C\subseteq C'$ and $D\subseteq D'$. 

Moreover, the monoids $J(M_C,N_D)$ (where $C\in \Sigma, D\in \Gamma$) form a monoidal complex $J(\Mcc,\Ncc)$ supported on $J(\Sigma, \Gamma)$. In fact, 
for each $C,C'\in \Sigma, D,D' \in \Gamma$ with $C\subseteq C',D\subseteq D'$, we have 
\[
J(M_{C'},N_{D'}) \cap J(C,D) =J(M_C,N_D).
\]
Finally, we will show that $k[\Mcc]\otimes_k k[\Ncc] \cong k[J(\Mcc,\Ncc)]$. Indeed, the isomorphism is given by 
\[
t^{m}\otimes t^{n} \mapsto t^{(m,n)}, ~m\in |\Mcc|, n\in |\Ncc|.
\]
Hence $k[\Mcc] \otimes_k k[\Ncc]$ is a toric face ring.
\end{proof}
\begin{cor}
\label{polynomial_extension_of_toric_face_rings}
Polynomial extensions of a toric face ring are also toric face rings.
\end{cor}
\begin{proof}
Apply Proposition \ref{tensor_product_of_toric_face_rings} for $\Mcc$ and $\Ncc=\N \subseteq \R$.
\end{proof}
\begin{prop}
\label{fiber_product_of_toric_face_rings}
The fiber product over $k$ of two toric face rings is a toric face ring.
\end{prop}
\begin{proof}
Recall that if $A=k[x_1,\ldots,x_n]/I$ and $B=k[y_1,\ldots,y_m]$ are $k$-algebras where $I\subseteq k[x_1,\ldots,x_n], 
J\subseteq k[y_1,\ldots,y_m]$ are ideals of polynomial rings, then the fiber product of $A$ and $B$ is the pull-back of the 
diagram $A \rightarrow k \leftarrow B$. Concretely, we have
\[
A \times_k B \cong k[x_1,\ldots,x_n,y_1,\ldots,y_m]/(I+J+(x_iy_j:1\le i\le n,1\le j \le m)).
\]
Let $\Mcc, \Ncc$ be monoidal complexes supported on the fans $\Sigma \subseteq \R^d, \Gamma \subseteq \R^e$ respectively. Embed 
$\R^d$ and $\R^e$ as vector subspaces into $\R^{d+e}$ such that $\R^d \cap \R^e=\{0\}$. We consider the union $\Sigma \cup \Gamma$ 
as a fan in $\R^{d+e}$. Consider the union $\Mcc \cup \Ncc$ as a monoidal complex supported on $\Sigma \cup \Gamma$. Then we have 
\[
k[\Mcc \cup \Ncc] \cong k[\Mcc] \times_k k[\Ncc].
\]
Hence $k[\Mcc] \times_k k[\Ncc]$ is a toric face ring.
\end{proof}
\begin{prop}
\label{Segre_toric_face_rings}
The Segre product of two homogeneous toric face rings over $k$ is again a toric face ring.
\end{prop}
\begin{proof}
Firstly, we note that the Segre product of two homogeneous affine monoid rings over $k$ is again an affine monoid ring over $k$. Concretely, let $M$ and $N$ are two homogeneous affine monoids, then the Segre product $k[M]*k[N]=k[M*N]$, where $M*N$ is the submonoid of the direct sum $M\times N$ consisting of elements $(u,v)$ with $\deg(u)=\deg(v)$.

We can generalize this construction for any two homogeneous toric face rings $k[\Mcc], k[\Ncc]$. For any two cones $C\in \Sigma, D\in \Gamma$, let $C*D$ be the cone generated by $M_C*N_D$. Let $\Sigma * \Gamma$ be the collection of cones $C*D$ and $\Mcc *\Ncc$ be the collection of affine monoids $M_C*N_D$ for such cones $C,D$. Then 
$\Mcc *\Ncc$ is a homogeneous monoidal complex supported on the fan $\Sigma * \Gamma$, and $k[\Mcc*\Ncc] \cong k[\Mcc]* k[\Ncc]$.
\end{proof}
\section{Betti numbers of the residue field}
\label{Bettinumbers}
Laudal and Sletsj\o{}e \cite[Prop.~1.3]{LS} computed Betti numbers
of affine monoid rings. The result was later reproved and extended in \cite[(1.2)]{PRS} and \cite[Thm.~2.1]{HRW}.
In this section, using the method employed in \cite{PRS} and \cite{HRW}, we find the Betti numbers of $k$ as a $k[\Mcc]$-module 
in the natural monoid grading (to be defined below).

Recall that $I$ is the defining ideal of $k[\Mcc]$ and $B_{\Mcc}$ is the ideal generated by pure binomials in $I$. 
\begin{defn}[The associated monoid]
Define a relation $\sim$ in $\N^n$: $a \sim b$ if and only if $X^a-X^b \in B_{\mathcal M}$. 
This relation is compatible with vector sum: $a\sim b$ implies $a+c \sim b+c$ for $a, b, c \in \N^n$. Let
$H$ be the monoid whose elements are equivalent classes of $\N^n/\sim$ and the addition inherited from that of $\N^n$. Then 
we call $H$ the monoid associated with $\Mcc$ and the system of generators $a_1,\ldots,a_n$.
\end{defn}
In general, $H$ depends both on $\Mcc$ and on the choice of the system of generators.
\begin{rem}
In Example \ref{asr}, the monoidal complex $\mathcal{M}$ is a positive affine monoid $M$, and we can choose $\{a_1,\ldots,a_n\}$ to be the minimal set of generators of $M$. Here $H=M$ and $B_{\mathcal M}$ is the toric ideal defining $k[M]$. In other words, the $H$-grading is the monoid grading induced by $M$. 
In Example \ref{Sta-Rei}, we have $B_{\mathcal M}=0$ and $H=\N^n$; the $H$-grading is simply the $\Z^n$-grading.
\end{rem}

We say that a monoid is {\em positive} if for elements $\lambda, \mu$ of this monoid with $\lambda+\mu=0$, we must have $\lambda=\mu=0$.
 We say that a monoid is {\em cancellative with respect to $0$} if an equation $\lambda+\mu=\lambda$ in the monoid implies that $\mu=0$. 
It is not hard to see that $H$ is a commutative positive monoid.
\newpage
\begin{lem}[\cite{BKR}, Lem.~4.4]
\label{monoid_H}
Denote the class of $a\in \N^n$ in $H$ by $\overline{a}$. We have:
\begin{enumerate}
\item 
If $\overline{a}+\overline{c}=\overline{b}+\overline{c}$ for $a,b,c \in \N^n$ then $X^a-X^b \in I$.
\item
$H$ is cancellative w.r.t.~ $0$.
\item
If $X^a-X^b \in I$ and $X^a, X^b \notin I$ then $\overline{a}=\overline{b}$ in $H$.
\end{enumerate}
\end{lem}
It is easy to see that $S=k[X_1,\ldots,X_n]$ and $R=k[\mathcal{M}]$ are $H$-graded. Note that $S/B_{\mathcal M}$ is exactly the monoid 
algebra $k[H]$ of $H$. 

Denote by $J$ the ideal $I/B_{\mathcal M}$ of $k[H]$. Let $e_g, g=1,\ldots,n$ be the standard basis vectors in $\R^n$. Then $J$ is
a {\em monomial ideal} of $k[H]$ in the sense that it is generated by elements $\overline{\sum _{g\in G}e_g}$ in $H$, where $G\subseteq [n]$
 such that $\prod_{ g\in G}a_g=0$. From Proposition \ref{def-ideal}, we get $k[{\mathcal M}]= k[H]/J$.

For elements $\lambda,\mu\in H$, we say that $\lambda < \mu$ if $\lambda\neq \mu$ and $\mu-\lambda\in H$. Then $(H,<)$ is a partially
 ordered set. It follows from \ref{monoid_H}(ii) that $<$ is a strict order.
\begin{defn}[The order complexes of divisor posets]
For each $\lambda \in H$, denote by $\Delta_{\lambda}$ the set of chains $\alpha_1 <\cdots <\alpha_i \in H$ such that
 $0=\alpha_0<\alpha_1$ and $\alpha _i < \lambda=\alpha_{i+1}$.

Let $\Delta _{\lambda, J}$ be the subset of $\Delta_{\lambda}$ consisting of chains $0=\alpha_0< \alpha_1 <\cdots
 <\alpha_i< \lambda=\alpha_{i+1}$ in $\Delta_{\lambda}$ such that for some $0\le j\le i$, the element $X^{\alpha _{j+1} -\alpha _j}$, as an element of 
 the group ring $k[H]$, belongs to $J$. 
 \end{defn}
For each $\lambda \in H$, the sets $\Delta_{\lambda}, \Delta _{\lambda, J}$ are simplicial complexes. Let $\widetilde{H}_{\ell}(
\Delta_{\lambda}, \Delta _{\lambda, J};k)$ be the $\ell$-th reduced simplicial homology with coefficients in $k$ of the pair
$(\Delta_{\lambda}, \Delta _{\lambda, J})$.

Since $\{a_1,\ldots,a_n\}$ is a homogeneous system of generators, there is a function
$|\cdot|: H \to \Z$ mapping $\lambda = \overline{a}\in H$ to $|\lambda|$, which is the sum of coordinates of $a$. This follows from the 
fact that $k[M_C]$ is standard graded for every $C\in \Sigma$.

Finally, we get the formula for graded Betti numbers of toric face rings, and a characterization of the Koszul property.
\begin{thm}
\label{Betti_numbers_Koszulness}
With the above notations, the bi-graded Betti number $\beta^{R}_{i, \lambda}(k)=\dim _k \Tor _i ^{R}(k,k)_\lambda$ of $k$ as an
$H$-graded module is given by
$$
\beta^{R}_{i, \lambda}(k) = \dim _{k} \widetilde{H}_{i-2}(\Delta_{\lambda}, \Delta _{\lambda, J};k),
$$
for every $\lambda \in H$ and $i>0$. 

In particular, the following statements are equivalent:
\begin{enumerate}
\item 
$k[\mathcal M]$ is a Koszul algebra;
\item
$\widetilde{H}_{i-2}(\Delta_{\lambda}, \Delta _{\lambda, J};k)=0$ for every $i>0$ and $\lambda \in H$ such that $|\lambda|>i.$
\end{enumerate}
\end{thm}
\begin{proof}
Similar to the proof of \cite[Thm.~2.1]{HRW}, we will exploit the bar resolution.

Consider the following free resolution of $k$ over $R=k[H]/J$, the so-called {\em bar resolution} (see Maclane \cite[Chap.~10, \S 2]{Ma}):
$$
\textbf{B}: \cdots \to B_i\to B_{i-1} \to \cdots\to B_0 =k.
$$
For each $i$, the module $B_i$ is the free $R$-module with basis elements $[\lambda_1|\cdots|\lambda_i]$, where
$\lambda_j \in (H \setminus J)-\{0\}$ for all $j$. The differential map $d_i:B_i\to B_{i-1}$ is defined on basis elements by the formula:
$$
d_i[\lambda_1|\cdots|\lambda_i] = X^{\lambda_1}[\lambda_2|\cdots|\lambda_i]+\sum _{j=1}^{i-1}(-1)^j[\lambda_1|\cdots|\lambda_j+
\lambda_{j+1}|\cdots|\lambda_i].
$$
Here we will understand the symbol $[\lambda_1|\cdots|\lambda_i]$ to be $0$ if one of the $\lambda_j$ is in $J$.

By definition, the Betti numbers $\Tor _{i}^R(k,k)$ are computed by the homology of $\textbf{B}\otimes k$. We can write the latter complex as
follows:
$$
\textbf{B}\otimes k: \cdots \to B_i\otimes k\to B_{i-1}\otimes k \to \cdots\to B_0\otimes k =k.
$$
The differential of this complex is given by:
$$
(d_i\otimes k)[\lambda_1|\cdots|\lambda_i]=\sum _{1\le j\le i-1}(-1)^j[\lambda_1|\cdots|\lambda_j+
\lambda_{j+1}|\cdots|\lambda_i].
$$
The important observation is that the differential map $d_i\otimes k$ preserves the sum $\lambda = \lambda_1 +\cdots+\lambda_i$. Hence
$\textbf{B}\otimes k$ is $H$-graded and the homology modules of $(\textbf{B}\otimes k)_{\lambda}$ compute $\Tor _{i}^R(k,k)
_{\lambda}.$

Since $<$ is a strict order on $H$, we can identify $[\lambda_1|\cdots|\lambda_i]$ with the chain $\lambda_1 <\lambda_1+\lambda_2
<\cdots<\lambda_1+\lambda_2+\cdots+\lambda_{i-1}$ in the open interval $(0,\lambda)$. Then the differential of the complex $(\textbf{B}\otimes k)_{\lambda}$ is
nothing but the boundary map of the reduced chain complex of pair $\widetilde{C}_{\pnt -2}(\Delta_{\lambda},\Delta_{\lambda, J};k)$. Thus
$$
\beta^{R}_{i, \lambda}(k) = \dim _{k} \widetilde{H}_{i-2}(\Delta_{\lambda}, \Delta _{\lambda, J};k),
$$
for all $i>0$ and $\lambda \in H$, as claimed.

The remaining characterization is obvious.
\end{proof}
\section{The quadratic condition}
\label{Koszul_conj}
We want to obtain necessary and sufficient conditions on $\mathcal M$ so that $R=k[\mathcal M]$ a Koszul algebra. Theorem 
\ref{Betti_numbers_Koszulness} is not easy to use, even for Stanley-Reisner ring, see \cite{HRW}. Instead of looking directly at 
relative homology of simplicial complexes, we turn our attention to the defining ideal of $k[\Mcc]$ and the monoid rings 
$k[M_C], C\in \Sigma$.

Firstly, we will relate the Koszul property of $R$ with the Koszul property of the involved affine monoid rings $k[M_C]$ where
 $C\in \Sigma$. As a corollary to Proposition \ref{Koszulness_under_retracts} we get:
\begin{prop}
\label{heredity}
If $k[\mathcal M]$ is Koszul then for any $C\in \Sigma$, the monoid ring $k[M_C]$ is Koszul.
\end{prop}
Now we come to an useful condition for the Koszul property of toric face rings.
\begin{defn}(Quadratic condition)
Given a homogeneous system of generators $a_1,\ldots,a_n$, $\Mcc$ is said to {\em satisfies the quadratic condition} if the monomial ideal $A_{\Mcc}$ is generated in degree $2$. In that case, we also say that $k[\Mcc]$ satisfies the quadratic condition.
\end{defn}
We immediately have the following interpretation of the quadratic condition.
\newpage
\begin{prop}
\label{interpretation_quadratic_condition}
The following statements are equivalent:
\begin{enumerate}
 \item The monoidal complex $\Mcc$ satisfies the quadratic condition w.r.t.~ the homogeneous system of generators $a_1,\ldots,a_n$.
 \item For any set $I\subseteq \{1,2,
\ldots,n\}$ such that the elements $\{a_i:i\in I\}$ do not belong to a common cone of $\Sigma$, there are two different indices $i,j\in I$ such that 
$a_i$ and $a_j$ do not belong to a common cone of $\Sigma$.
\end{enumerate}
\end{prop}
Clearly a monoidal complex which gives rise to a Stanley-Reisner rings defined by quadrics satisfies the quadratic condition. Moreover, the correspongding Stanley-Reisner ring 
is Koszul by Fr\"oberg's theorem. 
\begin{rem}
Without the quadratic condition, even when the ring $k[M_C]$ is a Koszul 
algebra for each facet $C\in \Sigma$ and the defining ideal $I$ is generated by quadrics, $k[\mathcal M]$ can be non-Koszul. 
See the example below.
\end{rem}
\begin{ex}
\label{quadratic_does_not_imply_Koszul}
Take $k = \Q$. Consider the points in $\R^4$ with the following coordinates
\begin{gather*}
A_1 = (2,0,0,0),\ A_2=(0,2,0,0),\ A_3 = (0,0,2,0),\\
 A_4 = (1,1,0,0),\ A_5 = (0,1,1,0).
\end{gather*}
Hence these points live in the affine subspace $\R^3=\{x_4=0\}$. The monoid ring generated by those 5 points is
 $k[X_1,\ldots,X_5]/I_1$ where $I_1 = (X_1X_2-X_4^2,X_2X_3-X_5^2)$. Let $O =(0,0,0,0)$ be the origin of $\R^4$.

Take the point $A_6= (0,0,0,2)$. Consider the fan in $\R^4$ with the following facets, which are simplicial 
cones: $\normalfont {OA_1A_2A_3}$, $\normalfont{OA_1A_3A_6}, \normalfont{OA_2A_6}$. 

\bigskip
\bigskip

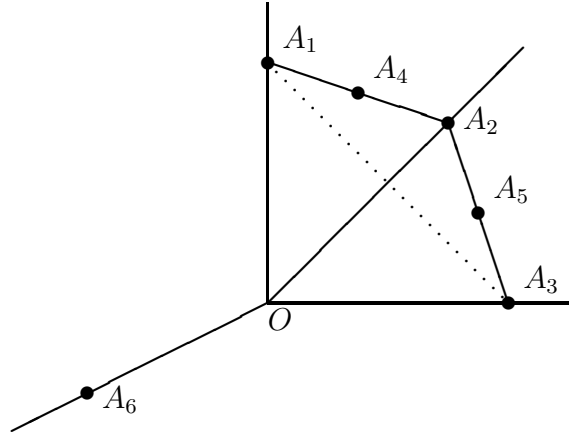
\setlength{\unitlength}{4cm}
\begin{picture}(1,1)
\thicklines
\put(1.5,0){\line(0,1){1}}
\put(1.5,0){\line(1,0){1}}
\put(1.5,0){\line(1,1){0.85}}
\put(1.5,0){\line(-2,-1){0.85}}
\put(1.5,-0.09){$O$}
\put(1.55,0.85){$A_1$}
\put(1.5,0.8){\circle*{0.04}} 
\put(2.35,0.05){$A_3$}
\put(2.3,0){\circle*{0.04}}
\put(2.15,0.58){$A_2$}
\put(2.1,0.6){\circle*{0.04}}
\put(1.85,0.75){$A_4$}
\put(1.8,0.7){\circle*{0.04}}
\put(2.25,0.35){$A_5$}
\put(2.2,0.3){\circle*{0.04}}
\put(0.95,-0.35){$A_6$}
\put(0.9,-0.3){\circle*{0.04}}

\put(1.5,0.8){\line(3,-1){.6}}
\put(2.3,0){\line(-1,3){.2}}
\multiput(1.5,0.8)(0.03,-0.03){27}{\circle*{0.01}}
\end{picture}

\bigskip
\bigskip
\bigskip
\bigskip
\bigskip
The (homogeneous) toric face ring we get is $R = k[X_1,\ldots,X_6]/I$ with $I = I_1 + (X_4X_6,X_5X_6)$.

The affine monoid rings supported on the $3$ maximal cones are Koszul. In fact, in the revlex order with $X_1<X_2<\cdots<X_5$, the polynomials
$\{X_1X_2-X_4^2,X_2X_3-X_5^2\}$ form a Gr\"obner basis for $I_1$. So $k[X_1,\ldots,X_5]/I_1$ is Koszul. The other two monoid rings are 
polynomial rings.

However, we can check by Macaulay2 \cite{GS} that $R$ is not a Koszul algebra, because $k$ has a non-linear second syzygy. The Betti table 
we get is the following.

\bigskip
\centering{
\begin{tabular}{r l l l l l l l l}
& 0 & 1 & 2 & 3 & 4 & 5 & 6 & 7 \\
total: & 1 & 6 & 19 & 46 & 101 & 217 & 468 & 1016 \\
0: & 1 & 6 & 19 & 45 & 92 & 173 & 309 & 534 \\
1: & . & . & . & 1 & 9 & 44 & 158 & 470\\
2: & . & . & . & . & . & . & 1 & 12\\
\end{tabular}
}
\bigskip

In this example, $A_{\Mcc}=(X_4X_6,X_5X_6,X_1X_2X_6,X_2X_3X_6)$ is not quadratic, so $\Mcc$ does not satisfy the quadratic condition w.r.t.~ the generators $A_1,\ldots,A_6$.
\end{ex}
The following conjecture was suggested by T. R\"{o}mer.
\begin{conj}[R\"{o}mer]
\label{main_conj}
Assume that $\Mcc$ satisfies the quadratic condition (with respect to a homogeneous system of generators), and for every cone $C\in \Sigma, k[M_C]$ is a Koszul algebra. Then
$k[\mathcal M]$ is Koszul.
\end{conj}
The quadratic condition behaves well under various constructions on toric face rings.
\begin{prop}
Let $k[\Mcc], k[\Ncc]$ be monoidal complexes satisfying the quadratic condition. Then the following toric face rings also satisfy the quadratic condition with respect 
to suitable system of generators:
\begin{enumerate}
 \item tensor product over $k$ of $k[\Mcc]$ and $k[\Ncc]$,
 \item fiber product over $k$ of $k[\Mcc]$ and $k[\Ncc]$,
 \item Segre product over $k$ of $k[\Mcc]$ and $k[\Ncc]$,
 \item Veronese subrings of $k[\Mcc]$,
 \item multigraded algebra retracts of $k[\Mcc]$.
\end{enumerate}
\end{prop}
\begin{proof}
From the hypothesis $\Mcc$ and $\Ncc$ satisfy the quadratic condition w.r.t.~ some system of generators. Assume that $a_1,\ldots,a_n$ is that homogeneous system of generators of $\Mcc$ and $b_1,\ldots,b_m$ is that of $\Ncc$. Let $k[X_1,\ldots,X_n]/I_{\Mcc}$ be the presentation of $k[\Mcc]$, $k[Y_1,\ldots,Y_m]/I_{\Ncc}$ the presentation of $k[\Ncc]$. 

(i) The monomial part of the defining ideal of $k[\Mcc]\otimes_k k[\Ncc]$ is $A_{\Mcc}+ A_{\Ncc}$, which is quadratic.

(ii) The monomial part of the defining ideal of $k[\Mcc \cup \Ncc]$ is $A_{\Mcc}+ A_{\Ncc}+ (X_iY_j:1\le i \le n, 1\le j\le m)$, which is also quadratic.

(iii) Recall from the proof of Proposition \ref{Segre_toric_face_rings} that $k[\Mcc]*k[\Ncc]=k[\Mcc*\Ncc]$. An element in the support of $\Mcc *\Ncc$ has the form $(a,b)$ where $a,b$ belongs to the support of $\Mcc$ and $\Ncc$ respectively, and $\deg (a)=\deg(b).$ Now assume that $(a_1,b_1),\ldots,(a_m,b_m)$ is a sequence of elements of $|\Mcc *\Ncc|$ such that $(a_1,b_1) \cdots (a_m,b_m)=0$. This implies that either $a_1\cdots a_m=0$ or $b_1\cdots b_m=0$. We may assume that the first equality holds. Since $\Mcc$ satisfies the quadratic condition, there exists some $i\neq j$ such that $a_ia_j=0$. But then 
$(a_i,b_i)\cdot (a_j,b_j)=0$. Hence $\Mcc*\Ncc$ also satisfies the quadratic condition.

(iv) Let $\Mcc^{(m)}$ be the $m$th Veronese of $\Mcc$. We show that $\Mcc^{(m)}$ satisfies 
the quadratic condition w.r.t. \{$\sum_{i=1}^n t_ia_i$: $\sum_{i=1}^n t_i=m, t_i\ge 0$\}. In fact, we will show that if $b_1\cdots b_p=0$ where $b_i\in |\Mcc^{(m)}|$ then $b_ib_j=0$ for some $i\neq j$.

We call $b_i$ extremal element if it is $b_i=ma_j$ for some $j$. We prove by induction on the number of non-extremal elements among $b_1,\ldots,b_p$.

If there is no non-extremal element, we are done by hypothesis on $\Mcc$.

Assume that $b_1$ is a non-extremal. Choose a cone $C\in \Sigma$ such that $b_1\in \relint C$. We can write $b_1=\sum_{i=1}^s t_ia_i$ where 
$\sum_{i=1}^s t_i=m, t_i> 0$ and $a_1,\ldots,a_s\in C$.

Clearly $(ma_1)\cdots (ma_s) b_2\cdots b_p=0$, hence by induction hypothesis on the number of non-extremal elements, there exist $2\le i<j\le p$ such that $b_ib_j=0$ (in 
this case we are done) or $ma_t b_i=0$ for $1\le t \le s$ and $2\le i \le p$. In the second case $b_i$ does not belong to 
any facet containing $C$, hence $b_1b_i=0$, hence we are also done.

(v)  From Proposition \ref{monoidal_complexes_retracts}, we know that any multigraded retract of $k[\Mcc]$ is isomorphic to some $k[\Mcc_{\Gamma}]$, where $\Mcc_{\Gamma}$ is a restricted subcomplex of $\Mcc$. It is not hard to check that the monomial part of the defining ideal of $k[\Mcc_{\Gamma}]$ is
\[
A_{\Mcc_{\Gamma}}= A_{\Mcc} \cap k[X_i: a_i \in |\Mcc_{\Gamma}|].
\]
Since $A_{\Mcc}$ is quadratic in $S=k[X_i:a_i \in |\Mcc|]$, we have $A_{\Mcc_{\Gamma}}$ is quadratic in 
$k[X_i: a_i \in |\Mcc_{\Gamma}|]$.
\end{proof}
These observations suggest the feasibility of the above conjecture.

We will provide in the following some evidences to support this conjecture, see \ref{quadGr}, \ref{strgKoszul},
 \ref{rmstrgKoszul} and \ref{i-K}. As a special case, if $k[\Mcc]$ is a Stanley-Reisner ring, this conjecture
 is confirmed by Fr\"oberg's theorem. In the case of affine monoid rings (where $\Sigma$ is the face poset of a cone),
 there is nothing to do.
\begin{rem}
On the other hand, it is not true that the ideal $A_{\mathcal M}$ is generated by quadratic monomials given $R$ being Koszul. See the following
 example \ref{non_monomial_quadratic}.
\end{rem}
\begin{ex}
\label{non_monomial_quadratic}
Continue with Example \ref{nontrivialtfr}. The toric face ring 
$$
k[\mathcal M]=k[X_1,X_2,X_3,X_4]/(X_1X_2-X_4^2, X_3X_4)
$$ 
is Koszul. Indeed, in the lex order, the set $\{X_1X_2-X_4^2, X_3X_4\}$ is a quadratic Gr\"obner basis for $I$.
However, the ideal $A_{\mathcal M}=(X_3X_4,X_1X_2X_3)$ is not generated by quadrics.
\end{ex}
The next result is similar to Bruns, Koch and R\"omer \cite[Prop.~3.2]{BKR}. We give a proof for the convenience of the 
reader.
\begin{prop}
\label{quadGr}
Let $<$ be a monomial order on $S=k[X_1,\ldots,X_n]$ and $<_i$ the induced order on $S_{C_i}=k[X_j|a_j\in M_{C_i}]$, where 
$C_1,\ldots,C_r$ are the facets of $\Sigma$. Then:
\begin{enumerate}
\item 
Assume that $k[\Mcc]$ is $G$-quadratic w.r.t.~ a monomial order $<$ on $S$. Then for $i=1,\ldots,r$, $k[M_{C_i}]$ is $G$-quadratic w.r.t.~ $<_i$.
\item
Assume that $\Mcc$ satisfies the quadratic condition. Additionally, assume that $<$ is a monomial order on $S$ such that 
the ideal $k[M_{C_i}]$ is $G$-quadratic w.r.t.~ $<_i$ for $i=1,\ldots,r$. Then 
$k[\Mcc]$ is $G$-quadratic w.r.t.~ $<$.
\end{enumerate}
\end{prop}
We will use the following simple lemma.
\begin{lem}
\label{Groebner_basis_of_binomial}
Let $I$ be an ideal of $S$ which is generated by monomials and binomials and $<$ a monomial order on $S$. Then with respect to 
$<$, the ideal $I$ has a Gr\"obner basis consisting of monomials and binomials. 
\end{lem}
\begin{proof}
We compute a Gr\"obner basis for $I$ starting with the monomials and binomials generating $I$. By using the Buchberger 
algorithm, at most we add new monomials or new binomials to the Gr\"obner basis of $I$. Hence the conclusion of the lemma follows.
\end{proof}
\begin{proof}[Proof of Proposition \ref{quadGr}]
The two statements follow from the formula
\[
\ini_<(I) = A_{\mathcal M} + \sum _{i=1}^rS\cdot \ini_{<_i}(I_{C_i}),
\]
and 
\[
\ini_{<_i}(I_{C_i})= S_{C_i}\cap \ini_<(I),
\]
for each $i=1,\ldots,r$. We will prove these formulae by using Lemma \ref{Groebner_basis_of_binomial}.

For the first formula, clearly the right-hand side is contained in the left-hand side. For the other inclusion,  
note that from Proposition \ref{def-ideal}, $I$ is generated by monomials in  and some binomials. 

Consider a Gr\"obner basis consisting of monomials and binomials of $I$. Clearly the initial forms of the monomials in 
question belong to $A_{\Mcc}$. Assume that $b$ is a proper binomial belongs to the Gr\"obner basis of $I$. Then either both 
monomials of $b$ are in $A_{\Mcc}$, which gives us nothing to do, or $b$ is in $I_{C_i}$ for some $i\in \{1,\ldots,r\}$. In the 
second case, again $\ini (b)$ belongs to the right-hand side of the first formula.

The second formula is proven in the same way.
\end{proof}
Example \ref{nontrivialtfr} shows that the converse of Proposition \ref{quadGr}.(ii) is not true. In fact, in this case the ring 
$k[\mathcal M]$ has a quadratic Gr\"obner basis w.r.t.~ the lex order, but $A_{\mathcal M}$ is not generated by quadrics.
\section{Strongly Koszul property}
\label{StronglyKoszul}
Recall the notion of strongly Koszul algebras.
\begin{defn}[Herzog, Hibi, Restuccia \cite{HHR}]
Let $R$ be a homogeneous $k$-algebra, $\mm=R_+$. Suppose that $a_1,\ldots,a_n \in R_1$ and minimally generate $\mm$. Then $R$ is 
called {\em strongly Koszul with respect to the sequence} $a_1,\ldots,a_n$ if for every $1\le i_1< \cdots < i_j\le n$, the ideal $(a_{i_1},
\ldots,a_{i_{j-1}}):a_{i_j}$ is generated by a subset of $\{a_1,\ldots,a_n\}.$
\end{defn}
The main result of this section characterizes strongly Koszul toric face rings.
\begin{thm}
\label{strgKoszul}
The following statements are equivalent:
\begin{enumerate}
\item 
 $k[\mathcal{M}]$ is strongly Koszul w.r.t.~ the sequence $\{a_1,\ldots, a_n\}$;
\item
\begin{enumerate}
\item 
for each $i = 1,\ldots,n$, we have $0:_{k[\mathcal{M}]}a_i = (a_{i_1},\ldots, a_{i_j})$, for some elements $a_{i_1},\ldots,
 a_{i_j}$ in $\{a_1,\ldots, a_n\}$,
\item
for each facet $C$ of $\Sigma,$ the ring $k[M_C]$ is strongly Koszul w.r.t.~ the sequence $\{a_1,\ldots, a_n\}\cap M_C$.
\end{enumerate}
\end{enumerate}
\end{thm}
Firstly, we have a simple lemma.
\begin{lem}
\label{colon-ideal}
Let $\II = (a_1,\ldots,a_{i-1}):_{R}a_i$. Let $C_1,\ldots,C_r$ be the facets of $\Sigma.$
For each $C \in \Sigma$, consider the following ideal of $k[M_C]:$
$$
\II _C =\begin{cases}
          (a_j|j<i, a_j\in M_C):_{k[M_C]}a_i & \text{if $a_i \in M_C$};\\
          0 & \text{if $a_i \notin M_C$.}
          \end{cases}
$$
Then:
\begin{enumerate}
\item
For each $C \in \Sigma$, we have $(a_1,\ldots,a_i)\cap k[M_C] = (a_j| j\le i, a_j \in M_C).$
\item
For each $C \in \Sigma$ such that $a_i \in M_C$, we have $\II \cap k[M_C] = \II _C.$
\item 
$\II = (0:_{R}a_i) + \sum _{l=1}^r R\cdot \II _{C_l}.$
\end{enumerate}
\end{lem}
\begin{proof}
This is a standard utilization of the $\Z^d$-grading. For (i), let $y \in (a_1,\ldots,a_i)\cap k[M_C]$ be homogeneous.
 If $a_j$ divides $y$ then because $y\in M_C$, we have $a_j \in M_C$. Thus (i) is proved.

For (ii), we see that $\II _C\subseteq \II \cap k[M_C].$ Let $z \in \II \cap k[M_C]$ be homogeneous. Since $za_i\in (a_1,\ldots,
a_{i-1})$, we have $za_i = wa_j$ where $j<i.$ But $wa_j \in M_{C}$, so  $w, a_j \in M_{C}$. Using the projection from $R$ to $k[M_C]$, we
 get $z\in \II _C$. 

The last part follows from (ii). We leave it as an exercise to the reader.
\end{proof}
\begin{proof}[Proof of Theorem \ref{strgKoszul}]
With Lemma \ref{colon-ideal}.(iii), it is immediate that (ii) $\Rightarrow$ (i).

Assuming that we have (i). Then (ii)(a) is clear.

Consider a facet $C$ of $\Sigma$, and a subsequence of $\{a_1,\ldots, a_n\}\cap M_C$. Without loss 
of generality we can assume this subsequence to be $a_1,\ldots,a_i$. 
By (i), we have the equality $(a_1,\ldots,a_{i-1}):_{k[\mathcal M]}a_i = (a_{i_1},\ldots,a_{i_k})$.

Using Lemma \ref{colon-ideal}(i), the ideal $(a_1,\ldots,a_{i-1}):_{k[M_C]}a_i = (a_{i_l}|a_{i_l}\in M_C).$ 
This implies (ii)(b).
\end{proof}
\begin{rem}
\label{rmstrgKoszul}
Part (a) of (ii) in the Theorem \ref{strgKoszul} is true if the $\Mcc$ satisfies the quadratic condition. However, the converse is not true as the next example demonstrates.
\end{rem}
\begin{ex} 
In $\R^3$ take six points with the following coordinates:
\begin{gather*}
A_1= (2,0,0),\ A_2=(0,2,0),\ A_3 =(0,0,2),\ A_4=(0,1,1),\\
 A_5=(1,0,1),\ A_6 =(1,1,0).
\end{gather*}
Consider the fan in $\R^3$ with three maximal cones $C_1, C_2, C_3$, where $C_1$ is generated by the points
$A_1,A_2,A_6$, the cone $C_2$ is generated by $A_3,A_1,A_5$ and the cone $C_3$ is generated by $A_2,A_3,A_4$ with the monoid 
relations $A_1A_2-A_6^2=A_2A_3-A_4^2=A_3A_1-A_5^2=0.$

Take $M_1$ to be generated by $A_1, A_2, A_6$ and in the same way, we have two other maximal monoids $M_2$ and $M_3$ of a monoidal complex $\Mcc$. 
The presentation ideal of $k[\Mcc]$ is 
\begin{gather*}
 I=(X_1X_2-X_6^2, X_2X_3-X_4^2,X_3X_1-X_5^2,X_1X_4,X_2X_5,X_3X_6,\\
X_4X_5,X_4X_6,X_5X_6).
\end{gather*}

\bigskip 

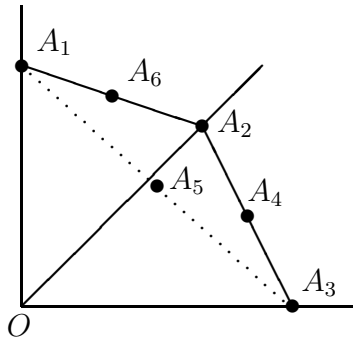
\setlength{\unitlength}{4cm}
\begin{picture}(1,1)
\thicklines
\put(1.5,0){\line(0,1){1}}
\put(1.5,0){\line(1,0){1.1}}
\put(1.5,0){\line(1,1){0.8}}
\put(1.45,-0.1){$O$}
\put(1.55,0.85){$A_1$}
\put(1.5,0.8){\circle*{0.04}} 
\put(2.43,0.05){$A_3$}
\put(2.4,0){\circle*{0.04}}
\put(2.15,0.58){$A_2$}
\put(2.1,0.6){\circle*{0.04}}
\put(1.85,0.75){$A_6$}
\put(1.8,0.7){\circle*{0.04}}
\put(2.25,0.35){$A_4$}
\put(2.25,0.3){\circle*{0.04}}
\put(1.99,0.39){$A_5$}
\put(1.95,0.4){\circle*{0.04}}

\put(1.5,0.8){\line(3,-1){.6}}
\put(2.4,0){\line(-1,2){.3}}
\multiput(1.5,0.8)(0.03,-0.027){30}{\circle*{0.01}}\\
\end{picture}

\bigskip
\bigskip
\bigskip

Because of Theorem \ref{strgKoszul}, the ring $k[\mathcal{M}]$ is strongly Koszul w.r.t.~ the sequence $X_1,X_2,\ldots,X_6$.
Indeed, part (b) of (ii) is true. For example, $k[M_1]$ is strongly Koszul w.r.t.~ the sequence $X_1,X_2,X_6$.

Part (a) if (ii) is true because of the following two identities:
\begin{enumerate}
\item 
$0:_{k[\mathcal{M}]}X_1 = (X_4)$
\item
$0:_{k[\mathcal{M}]}X_4= (X_1,X_5,X_6)$
\end{enumerate}
and four similar identities.

However, the monomial part
$$
A_{\Mcc}=(X_1X_2X_3,X_1X_4,X_2X_5,X_3X_6,X_4X_5,X_4X_6,X_5X_6)
$$ 
is not quadratic. 
\end{ex}
\section{Initially Koszul property}
\label{InitiallyKoszul}
In the followings, we consider initially and universally initially Koszul algebras in the sense of Blum \cite[Definition 1.3]{Bl}, 
Conca, Rossi and Valla \cite[Definition 2.2]{CRV}.
\begin{defn}
Let $R$ be a homogeneous $k$-algebra, and $a_1,\ldots,a_n \in R_1$ minimally generate $R_+$. The ring $R$ is called {\em initially Koszul}
(or {\em i-Koszul}) {\em with respect to the sequence} $a_1,\ldots,a_n$ if the set
$$
\mathcal{F} = \{(a_1,\ldots,a_i): i=0,\ldots,n\}
$$
is a Koszul filtration for $R$ in the sense of  Definition \ref{K_filtration}. In other words, for $i=1,\ldots,n$, we have
\[
(a_1,\ldots,a_{i-1}):a_i \in \Fc.
\]
We say that $R$ is {\em universally initially Koszul} (or {\em u-i-Koszul}) if $R$ is i-Koszul w.r.t.~ any $k$-basis of $R_1.$
\end{defn}
In this section, for a homogeneous quotient ring $R=k[X_1,\ldots,X_n]/I$, i-Koszulness means i-Koszulness w.r.t.~ the sequence of elements
$\overline{X_1},\ldots,\overline{X_n}.$ By \cite[Thm.~2.4]{CRV} and \cite[Thm.~2.1]{Bl}, i-Koszulness of $R$ implies that $I$ has quadratic Gr\"{o}bner basis in 
certain monomial order.

From \cite[Prop.~2.3]{Bl}, we see that a Stanley-Reisner ring $k[\Delta]$ is i-Koszul if and only if $\Delta$ is the full simplex.

The following theorem is the main result of this section.
\begin{thm}
\label{i-K}
If $k[\mathcal{M}]$ is i-Koszul w.r.t.~ the sequence $a_1,\ldots,a_n$, then $\Sigma$ is the face poset of a cone. In particular, 
$k[\mathcal{M}]$ is an affine monoid ring.
\end{thm}
In the proof of the theorem, we need the following lemma.
\begin{lem}
\label{i-Koszul}
Under the assumptions of Theorem \ref{i-K}, if $a_i$ generates an extremal ray of $\Sigma$ then $(a_1,\ldots,a_{i-1}):a_i= (a_1,\ldots,a_{i-1}).$
\end{lem}
\begin{proof} 
From the i-Koszulness assumption $(a_1,\ldots,a_{i-1}):a_i= (a_1,\ldots,a_j).$ If $j\ge i$, then
 $a_i \in (a_1,\ldots,a_{i-1}):a_i$. Thus $a_i^2=ba_l$ for some $l<i, b\in \cup _{C\in \Sigma}M_C$. Since $a_i^2\neq 0$, we conclude that 
$a_i, b, a_l$ belong to a common cone $C$ of $\Sigma$. Since $a_i$ generates an extremal ray of $C$, we must have $a_l$ belongs to that extremal ray. In other words, $a_l$ is a non-zero multiple of $a_i$. This contradicts with the minimality of $\{a_1,\ldots,a_n\}$.
\end{proof}
\begin{proof}[Proof of Theorem \ref{i-K}]
We only need to prove that all the extremal rays of $\Sigma$ belong to the same face of $\Sigma$. Assume that this is not the case. 
If $a_{i_1},\ldots,a_{i_t}$ generate the extremal rays of $\Sigma$ with $1\le i_1<\cdots<i_t \le n$ then $a_{i_1}a_{i_2}\cdots a_{i_t}=0$. 

Of course $a_{i_2}\cdots a_{i_t}\in (a_1,\ldots,a_{i_1-1}):a_{i_1}$. From Lemma \ref{i-Koszul},
we have that $a_{i_2}\cdots a_{i_t}\in (a_1,\ldots,a_{i_1-1})$. 
This in turn implies $a_{i_3}\cdots a_{i_t}\in (a_1,\ldots,a_{i_2-1}):a_{i_2}$. So again from Lemma \ref{i-Koszul}, we get
$a_{i_3}\cdots a_{i_t}\in (a_1,\ldots,a_{i_2-1})$. Iterate this argument and in the end we will have $a_{i_t} \in (a_1,\ldots,a_{i_{t-1}-1})$. 
This is a contradiction and hence we are done.
\end{proof}
\begin{cor}
If $k[\mathcal{M}]$ is a homogeneous u-i-Koszul toric face ring then
$k[\mathcal{M}]$ is a polynomial ring.
\end{cor}
\begin{proof} From Theorem \ref{i-K}, we get that $k[\mathcal{M}]$ is an affine monoid ring.
The Corollary follows from \cite[Prop.~5.5]{Bl}, which says that affine monoid rings which are u-i-Koszul must be
polynomial rings.
\end{proof}
\section*{Acknowledgment}
We are grateful to Tim R\"omer for generously suggesting problems and many insightful ideas on the subject of this paper. I am indebted to him 
for allowing me to include Conjecture \ref{main_conj} in the paper. We want to express our sincere thank to Ngo Viet Trung and Aldo Conca for 
their encouragement and inspiring comments. 

We thank the anonymous referee for his/her thoughtful comments which considerably helped us to improve the presentation of this paper.


\end{document}